\documentclass[12pt]{article}
\usepackage{amssymb,amsmath,amsthm,amscd,cite,hyperref, pgfplots, authblk, fancyhdr}
\usepackage[letterpaper, margin=1in]{geometry}

\title{Heat Kernel Estimates for Random Walks with Killing on Transient Graphs}

\author[1]{Emily Dautenhahn}
\author[2]{Anthony Graves-McCleary}

\affil[1]{Department of Mathematics, Gettysburg College, Gettysburg, PA, 17325 edautenh@gettysburg.edu}
\affil[2]{Department of Mathematics and Statistics, University of Alaska Fairbanks, Fairbanks, AK, 99775 ajgravesmccleary@alaska.edu}
\date{}

\newcommand{\set}[1]{\left\{#1\right\}}
\newcommand{\Real}{\mathbf R}

\newcommand{\N}{\mathbf{N}}

\newcommand{\Z}{\mathbf{Z}}

\newcommand{\abs}[1]{\left\vert#1\right\vert}
\newcommand{\norm}[1]{\left\Vert#1\right\Vert}

\newcommand{\fct}[3]{#1 \colon #2 \rightarrow #3}

\renewcommand{\phi}{\varphi}

\newcommand{\brac}[1]{\left< #1 \right>}

\newtheorem{theorem}{Theorem}[section]
\newtheorem{proposition}[theorem]{Proposition}
\newtheorem{lemma}[theorem]{Lemma}
\newtheorem{definition}[theorem]{Definition}
\newtheorem{example}[theorem]{Example}

\numberwithin{equation}{section}

\begin{document}
\maketitle

\begin{abstract}
We give matching two-sided estimates for heat kernels obtained from Schr\"odinger operators $\Delta+W$ on weighted graphs. The graphs we consider are transient and admit two-sided sub-Gaussian heat kernel estimates for the Laplacian $\Delta$, and the potentials $W$ decay to zero at infinity. We offer two proofs for one of our main results: one using the Feynman-Kac formula, and one without using it. Our results apply to fractal-type graphs such as graphical Sierpinski carpets, among other examples.\\

\noindent \textbf{Running Title:} Heat Kernel Estimates on Transient Graphs\\



\end{abstract}

\section{Introduction}\label{intro}

We are interested in estimating the heat kernels obtained from random walks with killing. This heat kernel may alternatively be interpreted as being obtained from a Schr\"odinger operator, i.e. a Laplacian plus a potential term $\Delta+W$.\\

The topic of heat kernel estimates on graphs has received extensive research interest; see the books by Barlow \cite{barlow1}, Grigor'yan \cite{grigoryan1}, Spitzer \cite{spitzer1}, Lawler and Limic \cite{lawlerlimic1}, and Woess \cite{woess1} for an introduction to this topic. Additionally, mathematical physicists have studied the spectra of Schr\"odinger operators $\Delta+W$ on graphs whose potentials $W$ are i.i.d. collections of random variables, in the context of quantum dynamics. See Aizenman and Warzel \cite{aizenmanwarzel1} for an overview of this area. Heat kernels of Schr\"odinger operators on other spaces such as Euclidean spaces and Riemannian manifolds has also received much attention; see for example Grigor'yan \cite{grigoryan2} and Takeda \cite{takeda1}.\\

 To our knowledge, the case of heat kernels of Schr\"odinger operators $\Delta+W$ on graphs whose potentials $W$ are deterministic remains relatively unexplored. Ancona \cite{ancona1} studied Green's functions of elliptic operators on graphs and obtained conditions for comparability between Green's functions. Our results have statements that are closely related to this, but our proofs are different with more general hypotheses on the underlying graph. Our main result of this paper is the following theorem. See Definitions \ref{hkbeta}, \ref{graphpotentialdef}, \ref{katoinfinity}, and \ref{subcriticaldef} for precise statements of the hypotheses.

\begin{theorem}[Heat kernel estimates]\label{heatkerneltheorem}

Let $(V, \mu)$ be a transient weighted graph satisfying sub-Gaussian heat kernel estimates HK($\beta$). Let $W$ be a graph potential in the Kato class $K^\infty(V, \mu)$. If $\Delta_\mu+W$ is subcritical, then there exist constants $c_1, c_2, c_3, c_4>0$ such that for all $x,y\in V$ and all $n\geq d(x,y)$, $$\frac{c_1}{V_\mu(x,n^{1/\beta})}\exp\left(-c_2\left(\frac{d(x,y)^\beta}{n}\right)^{1/(\beta-1)}\right)\leq p_\mu^W(x,y,n)\leq \frac{c_3}{V_\mu(x,n^{1/\beta})}\exp\left(-c_4\left(\frac{d(x,y)^\beta}{n}\right)^{1/(\beta-1)}\right).$$ 
    
\end{theorem}

The proof of this theorem hinges on our Theorem \ref{maintheorem}, which we prove using the iterated $h$-transform. We also offer an alternate proof of Theorem \ref{maintheorem} using results of Chen \cite{chen1} in Section 6.\\

Among the more general graphs to which our results apply are fractal graphs such as graphical Sierpinski carpets. Random walks without killing have been studied on such graphs by Barlow and Bass \cite{barlowbass2} and Grigor'yan and Telcs \cite{grigoryantelcs1}. These graphs and their random walks can be used in a limiting process to construct Brownian motion on fractals; see Kigami \cite{kigami1}, Kusuoka and Zhou \cite{kusuokazhou1} as well as Barlow, Bass, Kumagai and Teplyaev \cite{barlowbass3}. For more general results concerning fractals, see Barlow \cite{barlow3}, Graves-McCleary and Saloff-Coste \cite{gravesmccleary2}, and Grigor'yan and Hu \cite{grigoryanhu}.

\section{Graph Preliminaries}

\begin{definition}
A \textbf{weighted graph} is a pair $(V, \mu)$ where $V$ is a set and $\fct{\mu}{V\cup V^2}{[0, \infty)}$ is a function such that $\mu(x)>0$ for all $x\in V$, $\mu(x,y)=\mu(y,x)$ for all pairs $(x,y)\in V^2$ and for each $x\in V$, $\sum_{y\in V} \mu(x,y)= \mu(x).$
\end{definition}

In practice we let $\mu_{xy}$ denote $\mu(x,y)$. Note that we use $\mu$ to denote both edge weights and vertex weights. This will simplify our discussion of the $h$-transform in Section \ref{htrans}.  Elements $x\in V$ are called \textbf{vertices}, and if $\mu_{xy}>0$ we say that $x$ and $y$ are connected by an \textbf{edge}, denoted $x\simeq y$. If $x\simeq y$ and $x\neq y$ we write $x\sim y$.\\

Additionally, since $\sum_{y\in V}\mu_{xy}=\mu(x)$, one can elide defining $\mu_{xx}$ and define it later by setting $\mu_{xx}=\mu(x)-\sum_{y\sim x}\mu_{xy}$. We will do this for convenience when defining $h$-transformed weights later.

\begin{definition}
Let $\Omega\subseteq V$.

\noindent (i) A \textbf{path} in $\Omega$ is a sequence $x_0, \dots, x_N\in \Omega$ with $x_j\sim x_{j+1}$ for $j=0, \dots, N-1$. Given such a path, we define $N$ to be its \textbf{length}.\\

\noindent (ii) $\Omega$ is \textbf{connected} if for any $x,y\in \Omega$ there is a path $x_0, \dots, x_N$ with $x_0=x$ and $x_N=y$.\\

\noindent (iii) The (exterior) \textbf{boundary} of $\Omega$ is the set $$\partial \Omega=\set{x\notin \Omega\colon x\sim y \textrm{ for some }y\in \Omega}.$$
\end{definition}

\begin{definition}
Let $(V, \mu)$ be a weighted graph. The \textbf{distance} $d(x,y)$ between $x,y\in V$ is the length of the shortest path from $x$ to $y$. If there is no such path we say $d(x,y)=+\infty$.
\end{definition}

If $(V, \mu)$ is connected, $d$ is a metric on $V$ and $d(x,y)=1$ if and only if $x\sim y$.\\

We now introduce probability by defining the Markov kernel of a weighted graph.

\begin{definition}
Let $(V, \mu)$ be a weighted graph. The \textbf{Markov kernel} of $(V, \mu)$ is the function \begin{equation}K_\mu(x,y)=\frac{\mu_{xy}}{\mu(x)}.\end{equation} Given $\fct{f}{V}{\Real}$ we define \begin{equation}(K_\mu f)(x)=\sum_{y\in V} K_\mu(x,y)f(y).\end{equation}
\end{definition}

Therefore the Markov kernel can be viewed as an operator on functions on $V$. Since $\sum_{y\simeq x} \mu_{xy}=\mu(x)$ we have that $K_\mu 1=1$.

\begin{definition}
We say that the weighted graph $(V, \mu)$ has \textbf{controlled weights} if there exists $C>1$ such that for all $x\in V$ and $x\sim y$ we have \begin{equation} \frac{\mu_{xy}}{\mu(x)}\geq \frac{1}{C}.
\end{equation}
\end{definition}

\begin{definition}
We call the weighted graph $(V, \mu)$ \textbf{uniformly lazy} if there exists $c>0$ with $K_\mu(x,x)\geq c$ for all $x\in V$.
\end{definition}

In this paper: \textbf{we assume all our graphs are infinite, connected, have controlled weights and are uniformly lazy.} The phrase ``weighted graph" will be taken to mean ``uniformly lazy weighted infinite connected graph with controlled weights."\\

Next, define by recursion on $n\in \N$: for $n=0$, $K_{\mu, 0}(x,y)=1_{\set{x}}(y)$, and for $n\geq 1$, \begin{equation}K_{\mu,n+1}(x,y)=\sum_{z\in V} K_{\mu,n}(x,z)K_\mu(z,y).\end{equation} Note that this yields $K_{\mu, 1}=K_\mu$.

\begin{definition}
Let $(V, \mu)$ be a weighted graph. The \textbf{heat kernel} of $(V, \mu)$ is given for $x,y\in V$ and $n\in \N$ as \begin{equation}p_\mu(x,y,n)=\frac{K_{n,\mu}(x,y)}{\mu(y)}.\end{equation}
\end{definition}

Unlike the Markov kernel $K_{n,\mu},$ the heat kernel is always symmetric in $x$ and $y$.


\begin{definition}
Let $(V, \mu)$ be a weighted graph. The \textbf{Laplacian} of $(V, \mu)$, denoted $\Delta_\mu$, is the operator $$\Delta_\mu=I-K_\mu$$ where $I$ is the identity.
\end{definition}

The heat kernel satisfies the discrete heat equation
\[ \Delta_\mu p_\mu(x,y,n) = p_\mu(x,y,n) - p_\mu(x,y,n+1).\]

\begin{definition}
The \textbf{Green's function} of the weighted graph $(V,\mu)$ is defined by
\[ G_{\mu}(x,y) = \sum_{n=0}^\infty p_\mu(x,y,n).\]
    
\end{definition}

\begin{definition}
The weighted graph $(V,\mu)$ is said to be \textbf{transient} if $G_\mu(x,y) < +\infty$ for all $(x,y) \in V \times V$ (in particular, this includes all pairs $(x,x)$). Otherwise, $(V, \mu)$ is said to be \textbf{recurrent}. 
\end{definition}

There are several equivalent definitions of transience and recurrence; see for instance Theorem 1.8 of \cite{barlow1}. It is well-known that $\Z^d$ is recurrent if $d=1$ or $2$ and transient if $d \geq 3$.\\

Note that on a transient graph $(V, \mu)$, we have that if $y\in V$ is fixed and $u(x)=G_\mu(x,y)$, then $\Delta_\mu u=p_\mu(x,y,0)=\frac{1_{\set{y}}(x)}{\mu(y)}.$

\subsection{Notation}

We let $a\preceq b$ denote the existence of a constant $C>0$ such that $a\leq C\cdot b$. We let $a\asymp b$ denote that both $a\preceq b$ and $a\succeq b$ hold. Exact constants may sometimes change line to line. Given a number $a\in \Real$ we let $a_+=\max(a, 0)$ and $a_-=\max(-a, 0)$. Given a real valued function $f$ we let $f_+=\max(f, 0)$ and $f_-=\max(-f, 0)$ so that $f_+, f_-\geq 0$ and $f=f_+-f_-$. Given a weighted graph $(V, \mu)$ we let $$C(V)=\set{\fct{f}{V}{\Real}}$$ $$C_c(V)=\set{f\in C(V)\colon f\textrm{ has finite support}}$$ $$C_0(V)=\set{f\in C(V)\colon f\textrm{ vanishes at infinity}}.$$ 

On a weighted graph $(V, \mu)$, we let $B(x,r)$ denote the closed ball of center $x\in M$ and radius $r>0$. We let $V_\mu(x,r)=\mu(B(x,r))=\sum_{y\in B(x,r)}\mu(y)$. The distance on a weighted graph $(V, \mu)$ between $x$ and $y$ will always be denoted $d(x,y)$, and given a distinguished point $o\in V$ we let $\brac{x}=1+d(x,o)$ for $x\in V$.

\section{Heat Kernel Concepts}

\begin{definition}\label{hkbeta}
    Let $\beta\geq 2$. The weighted graph $(V, \mu)$ satisfies \textbf{sub-Gaussian heat kernel estimates with exponent $\beta$}, abbreviated \textbf{HK($\beta$)}, provided there exist constants $c_1, c_2, c_3, c_4>0$ such that for all $x,y\in V$ and all $n\geq d(x,y)$, $$\frac{c_1}{V_\mu(x,n^{1/\beta})}\exp\left(-c_2\left(\frac{d(x,y)^\beta}{n}\right)^{1/(\beta-1)}\right)\leq p_\mu(x,y,n)\leq \frac{c_3}{V_\mu(x,n^{1/\beta})}\exp\left(-c_4\left(\frac{d(x,y)^\beta}{n}\right)^{1/(\beta-1)}\right).$$ 
\end{definition}

The exponent $\beta$ is sometimes called the \textbf{walk dimension} or \textbf{escape exponent} of $(V, \mu)$. Delmotte \cite{delmotte1} proved that HK(2) is equivalent to the conjunction of the volume doubling property and the Poincar\'e inequality; see Delmotte's paper for more details and the definitions. Barlow and Bass \cite{barlowbass1} proved that HK($\beta$) is equivalent to volume doubling, a Poincar\'e inequality with exponent $\beta$, and a cutoff Sobolev inequality. This implies the following:

\begin{theorem}\label{bassbarlowthm} (Barlow-Bass \cite{barlowbass1}, Theorem 1.1)
    Let $(V, \mu)$ be a weighted graph and let $\nu$ be such that $\nu\asymp \mu$. Then $(V, \mu)$ satisfies HK($\beta$) if and only if $(V, \nu)$ does.
\end{theorem}

We now define a more specific version of the previous heat kernel estimates to address when $(V, \mu)$ has well-behaved volumes of balls.

\begin{definition}
Let $\beta\geq 2$ and $\alpha>0$. The weighted graph $(V, \mu)$ satisfies \textbf{HK$(\alpha, \beta)$} provided $(V, \mu)$ satisfies HK($\beta$) and also $V_\mu(x,r)\asymp r^\alpha$ over all $x\in V$ and all $r\geq 1$.
\end{definition}


\begin{lemma}\label{greenest0}
Let $(V, \mu)$ be a transient weighted graph satisfying HK$(\beta)$. Then we have that $$G_\mu(x,y)\asymp \sum_{n=d(x,y)}^\infty \frac{n^{\beta-1}}{V_\mu(x,n)}.$$
\end{lemma}

\noindent\textbf{Proof:} The proof can be adapted with minimal changes from the proof of Proposition 4.2 in Delmotte \cite{delmotte1}, which proves the statement for $\beta=2$.\qed\\

The next statement follows directly from Lemma \ref{greenest0}.

\begin{lemma}\label{greenest}
Let $(V, \mu)$ be a transient weighted graph satisfying HK$(\alpha,\beta)$. Then $\alpha>\beta$ and we have that $$G_\mu(x,y)\asymp \frac{1}{d(x,y)^{\alpha-\beta}}.$$
\end{lemma}



\section{Schr\"odinger Operators and Green's Functions}

\begin{definition}
Let $(V, \mu)$ be a weighted graph and let $\fct{W}{V}{\Real}$ be a real-valued function. The \textbf{Schr\"odinger operator associated to $W$} is the operator $\Delta_\mu+W$, where we view $W$ as a multiplication operator on functions on $V$.
\end{definition}

The function $W$ is called the \textbf{potential} of the Schr\"odinger operator.


\begin{definition}
Let $(V, \mu)$ be a weighted graph and let $\fct{W}{V}{\Real}$. A real-valued function $u$ is called a \textbf{profile for $(\Delta_\mu+W)$} if $u$ is defined on all of $V$, $u>0$, and $(\Delta_\mu+W)u=0$.
\end{definition}


For technical purposes we require the following restriction on our potentials.

\begin{definition}\label{graphpotentialdef}
Let $(V, \mu)$ be a weighted graph. A function $\fct{W}{V}{\Real}$ is called a \textbf{graph potential on $(V, \mu)$} if there exists $\epsilon>0$ such that for all $x\in V$, \begin{equation} W(x)\leq K_\mu(x,x)-\epsilon.\end{equation}
\end{definition}

If the context of the weighted graph $(V, \mu)$ is clear, we refer to a graph potential on $(V, \mu)$ as simply a graph potential. For a graph potential $W$ on $(V, \mu)$ we define the kernel $K^{ W}_\mu(x,y)$ by $$ K^{W}_\mu(x,y)=\begin{cases} 
      K_\mu(x,y), & x\neq y \\
      K_\mu(x,x)-W(x), & x=y.
   \end{cases}$$

Now the kernels $K_{\mu,n}^{W}(x,y)$, the heat kernel $p_\mu^W(x,y,n),$ and the Green function $G_{\mu}^W(x,y)$ are defined analogously to those given in Section \ref{intro}.\\

The definition of a graph potential is precisely what is required so that $K_\mu^W(x,x) \geq \epsilon,$ mirroring the uniformly lazy condition. If $W$ is any potential satisfying $0\leq W(x) <1$ for all $x \in V,$ then $K_{\mu}^W$ is a sub-Markovian kernel since $\sum_{y \in V} K_{\mu}^W(x,y) = 1 - W(x)\leq 1$.\\

Note that if $f$ is a real-valued function on $V$, then $$K_{\mu}^Wf(x)=\sum_{y\in V} K^{W}_\mu(x,y)f(y)=\left(\sum_{y\in V} K^\mu(x,y)f(y)\right)-W(x)f(x)$$ and therefore $$I-K_{\mu}^W=\Delta_\mu+W.$$ 


\begin{definition}\label{subcriticaldef}
    Let $(V, \mu)$ be a weighted graph and $W$ a graph potential on $(V, \mu)$. The Schr\"odinger operator $\Delta_\mu+W$ is called \textbf{subcritical} if the Green's function $G_\mu^W(x,y)$ exists, in other words if $$\sum_{n=0}^\infty p_\mu^W(x,y,n)<+\infty$$ for some (any) $x,y\in V$.
\end{definition}

Note that $\Delta_\mu$ itself is subcritical if and only if $(V, \mu)$ is transient.\\

We now have a local Harnack inequality for solutions to a Schr\"odinger equation.

\begin{lemma}
Let $(V, \mu)$ be a weighted graph and let $W$ be a graph potential on $(V, \mu)$. There exists a constant $C=C(V, \mu)>0$ such that for any $u\geq 0$ with $$(\Delta_\mu+W)u(x)=0$$ we have $$u(y)\leq Cu(x)$$  for $y\sim x$. 
\end{lemma}

\noindent\textbf{Proof:} From $$(\Delta_\mu+W)u(x)=0$$ we obtain $$u(x)+W(x)u(x)=K_\mu u$$ and therefore $$\big(1-K_\mu(x,x)+W(x)\big)u(x)=\sum_{y\sim x}K_\mu(x,y)u(y).$$ If $1-K_\mu(x,x)+W(x)=0$, this implies $u(y)=0$ for all $y\sim x$ and therefore any $C>0$ works. If $1-K_\mu(x,x)+W(x)<0$, we must have $u(x)=u(y)=0$ for all $y\sim x$ and therefore any $C>0$ works. Since $(V, \mu)$ has controlled weights, there exists $C_1>0$ such that $K_\mu(x,y)\geq 1/C_1$ for all $x\in V$ and all $y\sim x$. Also, we have that $1-K_\mu(x,x)+W(x)\leq 2$ since $W$ is a graph potential. Therefore, taking $$C=C(V, \mu)=2C_1$$ works, since $\big(1-K_\mu(x,x)+W(x)\big)u(x)=\sum_{y\sim x}K_\mu(x,y)u(y)$ implies $$K_\mu(x,y)u(y)\leq \big(1-K_\mu(x,x)+W(x)\big)u(x)\leq 2u(x)$$ and thus $u(y)\leq Cu(x)$.\qed\\

\noindent\textbf{Remark:} The proof above shows that existence of positive $u>0$ with $(\Delta_\mu+W)u=0$ on $V$ implies $1-K_\mu(x,x)+W(x)>0$ or $W(x)>K_\mu(x,x)-1$ for all $x\in V$.\\

One can iterate the previous lemma across paths to obtain the following:

\begin{lemma} \label{localharnackfull}
Let $\Omega\subseteq V$ be finite and connected and let $\fct{W}{\Omega}{\Real}$ be a graph potential. There exists $C=C(\Omega)>0$ such that for any $u\geq 0$ on $\Omega\cup\partial \Omega$ with $(\Delta_\mu+W)u=0$ on $\Omega$, we have $$\sup_\Omega u\leq C\inf_\Omega u.$$
\end{lemma}

Notice this implies that $u$ is either strictly positive on $\Omega$ or identically zero. Via a diagonalization procedure we obtain the following version of the Harnack principle:

\begin{lemma}\label{harnackprinciple}
Let $\Omega\subseteq V$ be connected and let $\fct{W}{\Omega}{\Real}$ be a graph potential. Let $(u_n)_{n\geq 0}$ be a sequence of functions such that $u_n\geq 0$ on $\Omega\cup \partial \Omega$ and $(\Delta_\mu+W)u_n=0$ on $\Omega$. Then either $u_n\rightarrow +\infty$ pointwise on $\Omega$ or there exists $u\geq 0$ with $(\Delta_\mu+W)u=0$ on $\Omega$ and a subsequence $(u_{n_k})_{k\geq 0}$ of $(u_n)_{n \geq 0}$ such that $u_{n_k}\rightarrow u$ pointwise on $\Omega$.\qed
\end{lemma}

\begin{lemma}
Let $(V, \mu)$ be a transient weighted graph and let $\Delta_\mu+W$ be subcritical. Then there exists a profile for $\Delta_\mu+W$, i.e. a function $h>0$ satisfying $$(\Delta_\mu+W)h=0.$$
\end{lemma}

\noindent\textbf{Proof:} Let $(x_n)$ be a sequence in $V$ tending to infinity. Let $$h_n(x):= \frac{G_\mu^W(x, x_n)}{G_\mu^W(x_0, x_n)}.$$ Note that $h_n(x_0)=1$ for all $n\geq 0$ and $(\Delta_\mu+W)h_n=0$ on $V\setminus \set{x_n}$. By the Harnack principle Lemma \ref{harnackprinciple}, $(h_n)_{n\geq 0}$ has a pointwise convergent subsequence. The pointwise limit $h$ must be a profile for $\Delta_\mu+W$.\qed\\

We now define a scale-invariant version of the Harnack inequality for solutions to Laplace's equation. This does not hold on all graphs, but has useful consequences.

\begin{definition}
A weighted graph $(V, \mu)$ satisfies the \textbf{scale-invariant elliptic Harnack inequality (EHI)} if there exists a constant $C_H$ such that for any ball $B$ and any non-negative $\Delta_\mu$-harmonic function $u$ defined on $2B$, we have $$\sup_B u\leq C_H\inf_B u.$$
\end{definition}

We now have the classical Liouville theorem.

\begin{lemma}\label{liouville}
(Liouville's Theorem) Let $(V, \mu)$ be a weighted graph satisfying (EHI). Let $u>0$ be a positive $\Delta_\mu$-harmonic function defined on all of $V$. Then $u$ is constant.
\end{lemma}

\noindent\textbf{Proof:} Let $w=u-\inf_V u$. This function is harmonic and non-negative. We also have $\inf_V w=0$. Fix $x\in M$ and note that $$\sup_{B(x,r)} w\leq C_H \inf_{B(x,r)} w$$ for any $r>0$. Letting $r\rightarrow +\infty$ we see that $\sup_{V} w=0$. Hence $u$ is constant.\qed\\

\begin{definition}[Green-bounded] Let $(V, \mu)$ be a transient weighted graph. A real-valued function $W$ on $V$ is \textbf{Green-bounded} on $(V, \mu)$ if $$\norm{W}_\mu:=\sup_{x\in V} \sum_{y\in V} \abs{W(y)}G_\mu(x,y)\mu(y)<+\infty.$$\end{definition}

\begin{definition}\label{katoinfinity}
Let $(V, \mu)$ be a transient weighted graph. For a real-valued function $W$ on $V,$ the \textbf{Kato infinity class} $K^\infty(V, \mu, \alpha)$ is defined as follows: $W\in K^\infty(V, \mu, \alpha)$ provided $$\lim_{R\rightarrow \infty} \sup_{x\in V} \sum_{y\in V\setminus B(o, R)} \abs{W(y)}G_\mu(x,y)\mu(y)<\alpha.$$ We also define $$K^\infty(V, \mu)=\bigcap_{s>0} K^\infty(V, \mu, s).$$
\end{definition}

\begin{lemma}\label{katodecay}
Let $(V, \mu)$ be a transient weighted graph satisfying HK($\alpha$, $\beta$). Fix $o\in M$ and for $x\in V$ set $\brac{x}=1+d(o, x)$. Assume that $\fct{W}{V}{\Real}$ satisfies $\abs{W(x)}\preceq \brac{x}^{-(\beta+\epsilon)}$ for some $\epsilon>0$. Then $W\in K^\infty(V, \mu)$.
\end{lemma}

\noindent\textbf{Proof:} Let $R>1$ and let $\Gamma=V\setminus B(o, R)$. Let $x\in V$. We have that $$\sum_{y\in \Gamma} G_\mu(x,y)\abs{W(y)}\mu(y)\preceq \sum_{y\in \Gamma} \frac{\brac{y}^{-(\beta+\epsilon)}}{d(x,y)^{\alpha-\beta}}\mu(y).$$ For $y\in \Gamma$ such that $d(y, o)\geq 2 d(x, o)$ we have that $$d(x,y)\geq d(y,o)-d(x,o)\geq d(y,o)-\frac{1}{2}d(y,o)\succeq \brac{y}.$$ 

Therefore for such $y\in \Gamma$ we have $$\frac{\brac{y}^{-(\beta+\epsilon)}}{d(x,y)^{\alpha-\beta}}\preceq \brac{y}^{-(\alpha+\epsilon)}.$$

We then have that $$\sum_{y\in \Gamma} \frac{\brac{y}^{-(\beta+\epsilon)}}{d(x,y)^{\alpha-\beta}}\mu(y)=\sum_{y\in \Gamma\setminus B(o, 2\brac{x})} \frac{\brac{y}^{-(\beta+\epsilon)}}{d(x,y)^{\alpha-\beta}}\mu(y)+\sum_{y\in \Gamma\cap B(o, 2\brac{x})}\frac{\brac{y}^{-(\beta+\epsilon)}}{d(x,y)^{\alpha-\beta}}\mu(y)$$ $$\preceq \sum_{y\in \Gamma\setminus B(o, 2\brac{x})} \brac{y}^{-(\alpha+\epsilon)}\mu(y)+\sum_{y\in \Gamma \cap B(o, 2\brac{x})}\frac{\brac{y}^{-(\beta+\epsilon)}}{d(x,y)^{\alpha-\beta}}\mu(y).$$ 

First, let $\mathfrak{L}\in \N$ be such that $2^{\mathfrak{L}}\leq R\leq 2^{\mathfrak{L}+1}$. We then have that $$\sum_{y\in \Gamma\setminus B(o, 2\brac{x})} \brac{y}^{-(\alpha+\epsilon)}\mu(y)\leq \sum_{y\in \Gamma} \brac{y}^{-(\alpha+\epsilon)}\mu(y)\preceq \sum_{n=\mathfrak{L}}^\infty \left(\sum_{y\in B(o, 2^{n+1})\setminus B(o, 2^n)}\brac{y}^{-(\alpha+\epsilon)}\mu(y)\right)$$ $$\asymp \sum_{n=\mathfrak{L}}^\infty \left(\sum_{y\in B(o, 2^{n+1})\setminus B(o, 2^n)}(2^n)^{-(\alpha+\epsilon)}\mu(y)\right)\preceq \sum_{n=\mathfrak{L}}^\infty (2^n)^{-\epsilon}\preceq R^{-\epsilon}.$$

 Next, if $\abs{x}\leq R/2$, then the next term is zero. Otherwise, $$\sum_{y\in \Gamma\cap B(o, 2\brac{x})}\frac{\brac{y}^{-(\beta+\epsilon)}}{d(x,y)^{\alpha-\beta}}\mu(y)=\sum_{y\in B(x, \frac{1}{2}\brac{x})}\frac{\brac{y}^{-(\beta+\epsilon)}}{d(x,y)^{\alpha-\beta}}\mu(y)+\sum_{y\in B(o, 2\brac{x})\setminus B(x, \frac{1}{2}\brac{x})}\frac{\brac{y}^{-(\beta+\epsilon)}}{d(x,y)^{\alpha-\beta}}\mu(y).$$ 

If $y\in B(x, \frac{1}{2}\brac{x})$, then $\brac{y}\asymp \brac{x}$. Letting $N\in \N$ with $2^N\leq \brac{x}\leq 2^{N+1}$ we then have that $$\sum_{y\in B(x, \frac{1}{2}\brac{x})}\frac{\brac{y}^{-(\beta+\epsilon)}}{d(x,y)^{\alpha-\beta}}\mu(y)\asymp \sum_{y\in B(x, \frac{1}{2}\brac{x})}\frac{\brac{x}^{-(\beta+\epsilon)}}{d(x,y)^{\alpha-\beta}}\mu(y)$$ $$\preceq \brac{x}^{-(\beta+\epsilon)}\sum_{n=0}^N \sum_{y\in B(x, 2^{n+1})\setminus B(x, 2^n)}\frac{\mu(y)}{d(x,y)^{\alpha-\beta}} \asymp \brac{x}^{-(\beta+\epsilon)}\sum_{n=0}^N \sum_{y\in B(x, 2^{n+1})\setminus B(x, 2^n)}\frac{\mu(y)}{(2^n)^{\alpha-\beta}}$$ $$\preceq\brac{x}^{-(\beta+\epsilon)}\sum_{n=0}^N \frac{(2^n)^\alpha}{(2^n)^{\alpha-\beta}}=\brac{x}^{-(\beta+\epsilon)}\sum_{n=0}^N (2^\beta)^n\asymp \brac{x}^{-(\beta+\epsilon)}(2^N)^\beta\asymp \brac{x}^{-\epsilon}\preceq R^{-\epsilon}.$$

Finally, if $y\in B(o, 2\brac{x})\setminus B(x, \frac{1}{2}\brac{x})$, then $d(x,y)\asymp \brac{x}$. Therefore letting $N$ be as before with $2^N\leq \brac{x}\leq 2^{N+1}$, we have $$\sum_{y\in B(o, 2\brac{x})\setminus B(x, \frac{1}{2}\brac{x})}\frac{\brac{y}^{-(\beta+\epsilon)}}{d(x,y)^{\alpha-\beta}}\mu(y)\asymp \brac{x}^{\beta-\alpha}\left(\sum_{y\in B(o, 2\brac{x})\setminus B(x, \frac{1}{2}\brac{x})}\brac{y}^{-(\beta+\epsilon)}\mu(y)\right)$$ $$\preceq \brac{x}^{\beta-\alpha}\sum_{n=0}^N \left(\sum_{y\in B(o, 2^{n+1})\setminus B(o, 2^{n+1})}\brac{y}^{-(\beta+\epsilon)}\mu(y)\right)\asymp \brac{x}^{\beta-\alpha}\sum_{n=0}^N \left(\sum_{y\in B(o, 2^{n+1})\setminus B(o, 2^{n+1})}(2^n)^{-(\beta+\epsilon)}\mu(y)\right)$$ $$\preceq \brac{x}^{\beta-\alpha}\sum_{n=0}^N (2^n)^{\alpha-(\beta+\epsilon)}$$ $$\asymp \brac{x}^{\beta-\alpha}(2^N)^{\alpha-(\beta+\epsilon)}\asymp \brac{x}^{-\epsilon}\preceq R^{-\epsilon}.$$

Putting all three pieces together, we find that $$\sum_{y\in V\setminus B(o, R)} G_\mu(x,y)\abs{W(y)}\mu(y)\preceq R^{-\epsilon},$$ with this estimate being uniform in $x$. This proves that $W\in K^\infty(V, \mu)$.\qed\\

\begin{lemma}\label{griglemma}
Let $(V, \mu)$ be a transient weighted graph and let $\fct{W}{V}{\Real}$ have finite support. Then $W$ is Green-bounded on $(V, \mu)$.
\end{lemma}

\noindent\textbf{Proof:} Let $\Omega\subseteq V$ be a finite connected set containing the support of $W$. It suffices to prove separately that $$\sup_{x\in \Omega\cup\partial \Omega}\sum_{y\in V}\abs{W(y)}G_\mu(x,y)\mu(y)<+\infty$$ and $$\sup_{x\notin \Omega\cup\partial \Omega}\sum_{y\in V}\abs{W(y)}G_\mu(x,y)\mu(y)<+\infty.$$

First note that $\Omega\cup\partial \Omega$ is finite, and since a supremum over a finite set of real numbers is finite, we have $$\sup_{x\in \Omega\cup\partial \Omega}\sum_{y\in V}\abs{W(y)}G_\mu(x,y)\mu(y)<+\infty.$$

Next, fix a distinguished point $o\in \Omega$. For $x\notin \Omega\cup \partial \Omega$, the function $u_x(y)=G_\mu(x,y)$ satisfies $\Delta_\mu u_x=0$ on $\Omega$. Therefore by the Harnack inequality of Lemma \ref{localharnackfull}, there exists $C>0$ such that for all $x\notin\Omega\cup\partial \Omega$, $$\sup_\Omega u_x\leq C\inf_\Omega u_x.$$ Hence $G_\mu(x,y)\asymp G_\mu(x, o)$ over all $x\notin \Omega\cup\partial \Omega$ and all $y\in \Omega$. Therefore for $x\notin \Omega\cup \partial \Omega$ we have $$\sum_{y\in V} \abs{W(y)}G_\mu(x,y)\mu(y)=\sum_{y\in \Omega} \abs{W(y)}G_\mu(x,y)\mu(y)\asymp G_\mu(x, o)\sum_{y\in \Omega} \abs{W(y)}\mu(y).$$ By Barlow \cite{barlow1} Theorem 1.31 we have that $G_\mu(x,o)\leq G_\mu(o,o)<+\infty$ for all $x\in V$. Therefore we have that $$\sup_{x\notin \Omega\cup\partial \Omega}\sum_{y\in V}\abs{W(y)}G_\mu(x,y)\mu(y) \preceq G_\mu(o,o)\sum_{y\in \Omega} \abs{W(y)}\mu(y)<+\infty$$ as desired.\qed

\begin{lemma}\label{smalls}
Let $(V, \mu)$ be a transient weighted graph and assume that $W$ is Green-bounded on $V$. If $\norm{W}_\mu<\frac{1}{2}$, then $\Delta_\mu+W$ has a profile $h>0$ such that $h\asymp 1$.
\end{lemma}

\noindent\textbf{Proof:} Define a bounded operator $\fct{T_W}{L^\infty(V, d\mu)}{L^\infty(V, d\mu)}$ by $$T_Wf(x)=\sum_{y\in V}W(y)G_\mu(x,y)f(y)\mu(y).$$ Note that $\norm{T_Wf}_\infty\leq \norm{W}_\mu\norm{f}_\infty$ for all $f\in L^\infty(V, d\mu)$. In particular, if $\norm{W}_\mu<1$, then $I+T_W$ is invertible by a Neumann series: $$(I+T_W)^{-1}=\sum_{k=0}^\infty (-T_W)^k.$$ If $h=(I+T_W)^{-1}1=\sum_{k=0}^\infty (-T_W)^k 1$, then $$\abs{h}\leq \sum_{k=0}^\infty \norm{W}_\mu^k<+\infty$$ and $$\abs{h}\geq 1-\sum_{k=1}^\infty \norm{W}_\mu^k=1-\frac{\norm{W}_\mu}{1-\norm{W}_\mu}>0.$$ Thus $h>0$ with $h\asymp 1$. We also have that $$h(x)=1-\sum_{y\in V} W(y)G_\mu(x,y)h(y)\mu(y)$$  which implies $\Delta_\mu h= -Wh$ or $$(\Delta_\mu +W)h=0.$$\qed

\section{Doob's $h$-transform}\label{htrans}

Let $(V, \mu)$ be a weighted graph and let $W$ be a graph potential on $(V, \mu)$. Assume $h>0$ is a function on $V$. Define new weights by $\nu(x)=h(x)^2\mu(x)$ for $x\in V$ and $\nu(x,y)=\nu_{xy}=h(x)h(y)\mu_{xy}$ for $x\neq y$ in $V$. This may or may not define a new weighted graph $(V, \nu)$ depending on the function $h$.\\

\noindent\textbf{Notation:} With $\nu$ defined in terms of $\mu$ and $h$ as above, we write $\nu=h^2\mu$ to denote the new $h$-transformed weights on $V$. This notation has been chosen to mimic that used for $h$-transforms on manifolds, but we caution that this has meaning for not only the vertices but also the edges of the weighted graph.\\

\begin{lemma} Let $(V, \mu)$ be a weighted graph and let $W$ be a graph potential on $(V, \mu).$ Let $h>0$ satisfy $(\Delta_\mu+W)h=0$. If $\nu=h^2\mu$, then $(V, \nu)$ is a weighted graph and $$K_\nu(x,y)=\frac{h(y)}{h(x)}K_\mu^W(x,y)$$ for all $x,y\in V$. In particular, $(V, \nu)$ is uniformly lazy and has controlled weights.

\end{lemma}

\noindent\textbf{Proof:} First we must show that $(V, \nu)$ is uniformly lazy, i.e. there exists $c>0$ such that for all $x\in V,$ $$1-\sum_{y\sim x}\frac{\nu_{xy}}{\nu(x)}\geq c.$$ Since $(\Delta_\mu+W)h=0$ we have that $$h(x)-\left(\sum_{y\simeq x}K_\mu(x,y)h(y)\right)+W(x)h(x)=0.$$

We have $$1-\sum_{y\sim x} \frac{\nu_{xy}}{\nu(x)}=1-\sum_{y\sim x}\frac{h(y)\mu_{xy}}{h(x)\mu(x)}$$ $$=\frac{1}{h(x)}\left(h(x)-\sum_{y\simeq x}K_\mu(x,y)h(y)\right)+K_\mu(x,x)$$ $$=\frac{1}{h(x)}\left(\Delta_\mu h(x)\right)+K_\mu(x,x)=-W(x)+K_\mu(x,x)=K_\mu^W(x,x).$$

Since $W$ is a graph potential, there exists $c>0$ with $-W(x)+K_\mu(x,x)\geq c$. Therefore the weighted graph $(V, \nu)$ is uniformly lazy. The fact that $(V, \nu)$ has controlled weights follows from the Harnack inequality applied to $h$. Namely, if $C_1>0$ such that $h(x)\leq  C_1h(y)$ for all $x\sim y$ and $C_2>0$ is such that $$\frac{\mu_{xy}}{\mu(x)}\geq \frac{1}{C_2}$$ for $x\sim y$, then for $x\sim y$ we have that $$\frac{\nu_{xy}}{\nu(x)}=\frac{h(y)\mu_{xy}}{h(x)\mu(x)}\geq \frac{1}{C_1C_2}.$$ Thus $(V, \nu)$ has controlled weights. Furthermore, we have the identity $$K_\nu(x,y)=\frac{h(y)}{h(x)}K_\mu^W(x,y)$$ which is valid for all $x,y\in V$.\qed

\begin{lemma}\label{classicalh}
With profile $h>0$ for $\Delta_\mu+W$ as above, and weights $\mu=h^2\nu$, we have that for $\fct{f}{V}{\Real}$, $$\Delta_\nu(f)=\frac{1}{h}(\Delta_\mu+W)(hf)$$ and for any $x,y\in V$, $n\in \N$, $$p_\mu^W(x,y,n)=h(x)h(y)p_\nu(x,y,n).$$
\end{lemma} 

\noindent\textbf{Proof:} For $\fct{f}{V}{\Real}$, it is straightforward to verify that $$\Delta_\nu(f)=\frac{1}{h}(\Delta_\mu+W)(hf).$$

Next, a straightforward induction shows that $$K_{\nu, n}(x,y)=\frac{h(y)}{h(x)}K_{\mu, n}^W(x,y)$$ for all $x,y\in V$. Therefore the heat kernel of $\Delta_\nu$ is given by $$p_\nu(x,y,n)=\frac{K_{\nu, n}(x,y)}{\nu(y)}=\frac{p_\mu^W(x,y,n)}{h(x)h(y)}.$$\qed

\subsection{Iterated $h$-transform}

In this section we discuss the iterated $h$-transform. Our presentation is adapted from Graves-McCleary and Saloff-Coste \cite{gravesmccleary1}, which treats manifolds instead of graphs.\\

\begin{lemma}\label{graphpotentials}
Let $(V, \mu)$ be a weighted graph and let $\fct{W_1, W_2}{V}{\Real}$ be such that $W_1$ and $W_1+W_2$ are graph potentials on $(V, \mu)$. Let $h>0$ be such that $(\Delta_\mu+W)h=0$ and let $\nu=h^2\mu$. Then $W_2$ is a graph potential on $(V, \nu)$.
\end{lemma}

\noindent\textbf{Proof:} There exists $\epsilon>0$ such that $W_1(x)+W_2(x)\leq K_\mu(x,x)-\epsilon$ for all $x\in V$. We then have that $W_2(x)\leq K_\mu(x,x)-W_1(x)-\epsilon=K_\nu(x,x)-\epsilon$ for all $x\in V$, proving that $W_2$ is a graph potential on $(V, \nu)$.\qed\\

\begin{lemma}\label{iter} Let $(V, \mu)$ be a weighted graph and let $\fct{W_1, W_2}{V}{\Real}$ be such that $W_1$ and $W_1+W_2$ are graph potentials on $(V, \mu)$. Let $g,h>0$ be positive functions on $V$ satisfying $(\Delta_\mu+W_1)h=0$ and $(\Delta_\nu+W_2)g=0$, where $\nu=h^2\mu$. Then letting $\theta=g^2h^2\mu$, we have that \begin{equation}
p^{W_1+W_2}_\mu(x,y, n)=g(x)h(x)g(y)h(y)p_\theta(x,y,n)
\end{equation} for all $t>0$ and $x,y\in M$.
\end{lemma}

\noindent\textbf{Proof:} Note that \begin{equation}
(\Delta_\mu+W_1+W_2)(gh)=h\Delta_\nu g+W_2gh=h(\Delta_\nu +W_2)g=0.
\end{equation} Thus $gh$ is a profile for $\Delta_\mu+W_1+W_2$. The conclusion then follows from Lemma \ref{classicalh}.\qed\\

\begin{lemma}\label{itergreen}
Let $(V, \mu)$ be a weighted graph and let $\fct{W_1, W_2}{V}{\Real}$ be such that $W_1$ and $W_1+W_2$ are graph potentials on $(V, \mu)$. Let $h>0$ satsify $(\Delta_\mu+W_1)h=0$ and let $\nu=h^2\mu$. Let $g>0$ satisfy $(\Delta_\nu+W_2)g=0$ and let $\theta=g^2\nu$. Then the following are equivalent:

\noindent(a) $\Delta_\mu+W_1+W_2$ is subcritical.

\noindent(b) $\Delta_\nu+W_2$ is subcritical.

\noindent(c) $(M, \theta, K_\theta)$ is transient.
\end{lemma}

\noindent\textbf{Proof:} Follows directly from summing Lemma \ref{iter} over $n\in \N$.\qed\\

\begin{lemma}\label{unique}
Let $(V, \mu)$ be a weighted graph satisfying HK($\beta$). Let $\fct{W}{V}{\Real}$ be a graph potential such that $\Delta_\mu+W$ admits a profile $h>0$ with $h\asymp 1$. Let $g>0$ be another profile of $\Delta_\mu+W$. Then $g/h$ is constant.
\end{lemma}

\noindent\textbf{Proof:} Let $\nu=h^2\mu$. Then $(V, \nu)$ satisfies HK($\beta$) and therefore (EHI). We have that $g/h$ is positive and $\Delta_\nu$-harmonic. By Liouville's theorem, $g/h$ is constant.\qed

\begin{proposition}\label{compactstep}
Let $(V, \mu)$ be a transient weighted graph satifying HK($\beta$). Let $W$ be a graph potential with finite support such that $\Delta_\mu+tW$ is subcritical for all $0\leq t\leq 1$. Then $\Delta_\mu+W$ has a profile $h$ with $$h\asymp 1.$$
\end{proposition}

\noindent\textbf{Proof:} Consider the set $J=\set{0\leq t\leq 1\colon \Delta_\mu+tW \textrm{ has a profile }h\asymp 1}$. Assume for a contradiction that $[0,1]\setminus J$ is nonempty and let $s=\inf ([0,1]\setminus J).$ By Lemma \ref{smalls}, $s>0$.\\

Let $h>0$ be a profile for $\Delta_\mu+sW$ and let $\nu=h^2\mu$. Since $\Delta_\mu+sW$ is subcritical, by Lemma \ref{itergreen}, $(V, \nu)$ is transient.\\

By Lemma \ref{griglemma}, $W$ is Green-bounded on $(M, \nu)$. Therefore by Lemma \ref{smalls} there exists $\epsilon>0$ such that $s-\epsilon>0$ and $\Delta_\nu-\epsilon W$ has a profile $g>0$ with $g\asymp 1$. It follows that $gh\asymp h$ is a profile for $\Delta_\mu+(s-\epsilon)W$. By Lemma \ref{unique}, $gh\asymp 1$ and therefore $h\asymp 1$, a contradiction. Thus $J=[0,1]$.\qed\\

The next step eliminates the hypothesis that $\Delta_\mu+tW$ is subcritical for all $0\leq t\leq 1$ and replaces it with the simpler hypothesis that $\Delta_\mu+W$ is subcritical.

\begin{lemma}\label{compactstep2}
Let $(V, \mu)$ be a transient weighted graph satifying HK($\beta$). Let $W$ be a graph potential with finite support be such that $\Delta_\mu+W$ is subcritical. Then $\Delta_\mu+W$ has a profile $h$ with $$h\asymp 1.$$
\end{lemma}

\noindent\textbf{Proof:} Step 1: Assume $W\geq 0$. Here we have that $\Delta_\mu+tW$ is subcritical for all $0\leq t\leq 1$ so Proposition \ref{compactstep} implies the conclusion.\\

Step 2: General case. Here write $W=W_+-W_-$ where $W_+, W_-\geq 0$ are of finite support. Then by Step 1, $\Delta_\mu+W_+$ has a profile $h>0$ with $h\asymp 1$. Let $\nu=h^2\mu$ and consider the graph $(V, \nu, K_\nu)$. Since $(V, \mu, K_\mu)$ satisfies HK($\beta$), so does $(V, \nu, K_\nu)$.\\

Further, $\Delta_\nu-W_-$ is subcritical by Lemma \ref{itergreen}, and in fact since $-tW_-\geq -W_-$ for all $0\leq t\leq 1$, $\Delta_\nu-tW_-$ is subcritical for all $0\leq t\leq 1$. Therefore by Proposition \ref{compactstep}, $\Delta_\nu-W_-$ has a profile $g>0$ with $g\asymp 1$.\\

Finally, $gh\asymp 1$ is a profile for $\Delta_\mu+W$, as desired.\qed\\

\begin{theorem}\label{maintheorem}
Let $(V, \mu)$ be a transient weighted graph satisfying HK($\beta$). Let $W\in K^\infty(V, \mu)$ be a graph potential. If $\Delta_\mu+W$ is subcritical, then there exists a profile $h>0$ for $\Delta_\mu+W$ satisfying $$h\asymp 1.$$
\end{theorem}

\noindent\textbf{Proof:} Since $W\in K^\infty(V, \mu)$, we can write $W=W_1+W_2$ where $W_2$ has finite support disjoint from that of $W_1$ and $$\sup_{x\in V}\sum_{y\in V} G_\mu(x,y)\abs{W_1(y)}\mu(y)<\frac{1}{2}.$$ We can also take $W_1$ to be a graph potential. By Lemma \ref{smalls}, there exists a profile $g>0$ for $\Delta_\mu+W_1$ with $g\asymp 1$. Let $\nu=g^2 \mu$. Then $(V, \nu)$ satisfies HK($\beta$), and by Lemma \ref{itergreen} $\Delta_\nu+W_2$ is subcritical.\\

Since $W_2$ has finite support, by Lemma \ref{compactstep2}, $\Delta_\nu+W_2$ has a profile $f>0$ with $f\asymp 1$. Therefore $h=gf$ is a profile for $\Delta_\mu+W$ satisfying $h\asymp 1$.\qed

\section{Continuous Time Random Walks}

In this section we discuss random walks on graphs with continuous time. See Barlow \cite{barlow1} for an in-depth treatment of the topic. Continuous time random walks offer an alternative proof of Theorem \ref{maintheorem} via results of Chen \cite{chen1}. The rest of the paper does not depend on results from this section.\\

Let $(X_n)_{n\geq 0}$ denote the discrete time random walk on the weighted graph $(V, \mu)$. Let $(N_t)_{t\geq 0}$ be an independent Poisson process with rate 1. We define the continuous time random walk for $t\geq 0$ by $$Y_t=X_{N_t}.$$ This means that the continuous random walk stays at a vertex $x$ for time $\sim \exp(1)$, then jumps to a neighbor with the same distribution as the discrete time random walk.\\

The process $(Y_t)_{t\geq 0}$ is a strong Markov process with generator $\Delta_\mu$ that is right continuous and has left limits. In particular, the results of Chen \cite{chen1} apply.\\

The Schr\"odinger operator $\Delta_\mu+W$ has a continuous time heat semigroup denoted $Q_{t, \mu}^{W}$ for $t>0$. In the case $W=0$ we denote the continuous time heat semigroup of $\Delta_\mu$ by $Q_{t, \mu}$ for $t>0$. These semigroups have corresponding heat kernels $q_\mu^W(x,y,t)$ and $q_\mu(x,y,t)$.\\

If $h>0$ is a profile of $\Delta_\mu+W$, we can define the $h$-transformed semigroup $Q_{t,\nu}$ of $Q_{t, \mu}^W$ by $$Q_{t, \nu}[f]=\frac{1}{h}Q_{t, \mu}^W[hf].$$ Then we have that $$q_\nu(x,y,t)=\frac{Q_{t,\nu}[1_{\set{x}}]}{\nu(y)}=\frac{\frac{1}{h(x)}Q_{t, \mu}^W[h1_{\set{x}}]}{h(y)^2\mu(y)}=\frac{q_\mu^W(x,y,t)h(y)\mu(y)}{h(x)h(y)^2\mu(y)}=\frac{q_\mu^W(x,y,t)}{h(x)h(y)}.$$

\begin{lemma}
Assume that the Schr\"odinger operator $\Delta_\mu+W$ is subcritical with respect to the discrete time random walk. Then $\Delta_\mu+W$ is subcritical with respect to the continuous time random walk.
\end{lemma}

\noindent\textbf{Proof:} Let $h>0$ be a profile for $\Delta_\mu+W$ and let $\nu=h^2\mu$. From Barlow \cite{barlow1}, $$q_\nu(x,y,t)=\sum_{n=0}^\infty \frac{e^{-t}t^n}{n!}p_\nu(x,y,n).$$ It follows from Lemma \ref{classicalh} that $$q_\mu^W(x,y,t)=\sum_{n=0}^\infty \frac{e^{-t}t^n}{n!}p_\mu^W(x,y,n).$$ We therefore have that $$\int_0^\infty q_\mu^W(x,y,t)dt=\sum_{n=0}^\infty p_\mu^W(x,y,n),$$ and the result follows.\qed\\

By Theorem 3.6 of Chen \cite{chen1}, when $\Delta_\mu+W$ is subcritical, then it is gaugeable, i.e. $$\sup_{x\in V} E^x\left[\exp\left(-\int_0^\infty W(Y_s)ds\right)\right]<+\infty.$$

We can now define the Feynman-Kac gauge by $$h(x)=E^x\left[\exp\left(-\int_0^\infty W(Y_s)ds\right)\right].$$

\begin{lemma}
Let $(V, \mu)$ be a transient weighted graph and assume that $W$ is Green-bounded on $(V, \mu)$. Assume that the Schr\"odinger operator $\Delta_\mu+W$ is subcritical on $(V, \mu)$. Then the Feynman-Kac gauge $h$ is a profile for $\Delta_\mu+W$.
\end{lemma}

\noindent\textbf{Proof:} Note that for $x\in V$, we have $$\sum_{y\in V}G_\mu(x,y)\abs{W(y)}\mu(y)=\sum_{y\in V}\left(\int_0^\infty q_\mu(x,y,t)dt\right) \abs{W(y)}\mu(y)$$ $$=\int_0^\infty \sum_{y\in V} q_\mu(x,y,t)\abs{W(y)}\mu(y)dt= \int_0^\infty E^x\left[ \abs{W(Y_t)}\right]dt=E^x\left[\int_0^\infty \abs{W(Y_s)}ds\right].$$ Therefore $\sup_{x\in V} E^x\left[\int_0^\infty \abs{W(Y_s)}ds\right]=C<+\infty$. By Jensen's inequality, we have that $$h(x)=E^x\left[\exp\left(-\int_0^\infty W(Y_s)ds\right)\right]\geq \exp\left(-E^x\left[\int_0^\infty W(Y_s)ds\right]\right)\geq e^{-C}>0.$$ Thus $h$ is strictly positive. Furthermore, if $x\in V$ and $\tau\sim \exp(1)$ denotes the first jumping time $\tau=\inf\set{t\colon N_t>0}$, then by the Markov property, $$h(x)=E^x\left[\exp\left(-\int_0^\infty W(Y_s)ds\right)\right]$$ $$=E^x\left[\exp\left(-\int_0^\tau W(Y_s)ds\right)E^{Y_\tau}\left[\exp\left(-\int_0^\infty W(Y_s)ds\right)\right]\right]$$ $$=E^x\left[\exp\left(-\int_0^\tau W(Y_s)ds\right)h(Y_\tau)\right]=E^x\left[e^{-\tau W(x)}h(Y_\tau)\right].$$ Since $e^{-\tau W(x)}$ and $h(Y_\tau)$ are independent, we have that $$h(x)=E^x\left[e^{-\tau W(x)}h(Y_\tau)\right]=E^x\left[e^{-\tau W(x)}\right]E^x\left[h(Y_\tau)\right]=\frac{1}{1+W(x)}(K_\mu h)(x).$$ We therefore conclude that $$h+Wh=K_\mu h$$ or $$(\Delta_\mu+W)h=0.$$ Thus $h$ is a profile for $\Delta_\mu+W$.\qed\\

We now present a second proof of Theorem \ref{maintheorem}.\\

\noindent\textbf{Second Proof of Theorem \ref{maintheorem}:} Let $(V, \mu)$ be a transient weighted graph and let $W\in K^\infty(V, \mu)$ be such that $\Delta_\mu+W$ is subcritical. By Theorem 3.6 of Chen \cite{chen1}, $W$ is gaugeable, meaning that $$h(x)=E^x\left[\exp\left(-\int_0^\infty W(Y_s)ds\right)\right]$$ provides a profile for $\Delta_\mu+W$ such that $h\asymp 1$.\qed

\section{Heat Kernel Estimates: Proof of Theorem \ref{heatkerneltheorem}}

We now apply our results to obtain heat kernel estimates for Schr\"odinger operators, and in particular prove our main result Theorem \ref{heatkerneltheorem}.\\

\noindent\textbf{Proof of Theorem \ref{heatkerneltheorem}:} By Theorem \ref{maintheorem}, there exists a profile $h>0$ for $\Delta_\mu+W$ satisfying $h\asymp 1$. Let $\nu=h^2\mu$.\\

By Theorem \ref{bassbarlowthm}, the weighted graph $(V, \nu)$ satisfies HK($\beta$). The conclusion then follows from Lemma \ref{classicalh}.\qed

\subsection{Examples}

We now turn to examples of graphs and potentials satisfying the hypothesis of Theorem \ref{heatkerneltheorem}.

\begin{example}
    Lattice $\Z^d$
\end{example}

\noindent Here the edge relation is given by $\mu_{xy}=1$ if and only if $\norm{x-y}_2\leq 1$. It is well known that $(\Z^d, \mu)$ satisfies HK($d$, 2). $(\Z^d, \mu)$ is recurrent for $d=1,2$ and transient for $d\geq 3$.\\

By the results of Delmotte \cite{delmotte1}, if $\nu$ is any edge weight function such that $\nu\asymp \mu$, with $\mu$ as above, then $(\Z^d, \nu)$ satisfies HK($d$, 2). For $d\geq 3$ and any $\epsilon>0$, if $\fct{W}{\Z^d}{\Real}$ is a graph potential satisfying $\abs{W(x)}\leq (1+\abs{x})^{-(2+\epsilon)}$ and $\Delta_\nu+W$ is subcritical, then Theorem \ref{heatkerneltheorem} applies and the heat kernel $p_\nu^W(x,y,n)$ satisfies two-sided Gaussian-type estimates of the form HK($d$, 2).

\begin{example}
Discrete Heisenberg Group
\end{example}

The discrete Heisenberg group $H(\Z)$ consists of $3\times 3$ integer matrices $$\begin{pmatrix}
1 & x & y\\
0 & 1 & z\\
0 & 0& 1
\end{pmatrix},\hspace{2mm} x,y,z\in \Z.$$

\begin{figure}[hbt!]
\begin{centering}
\includegraphics[scale=0.65]{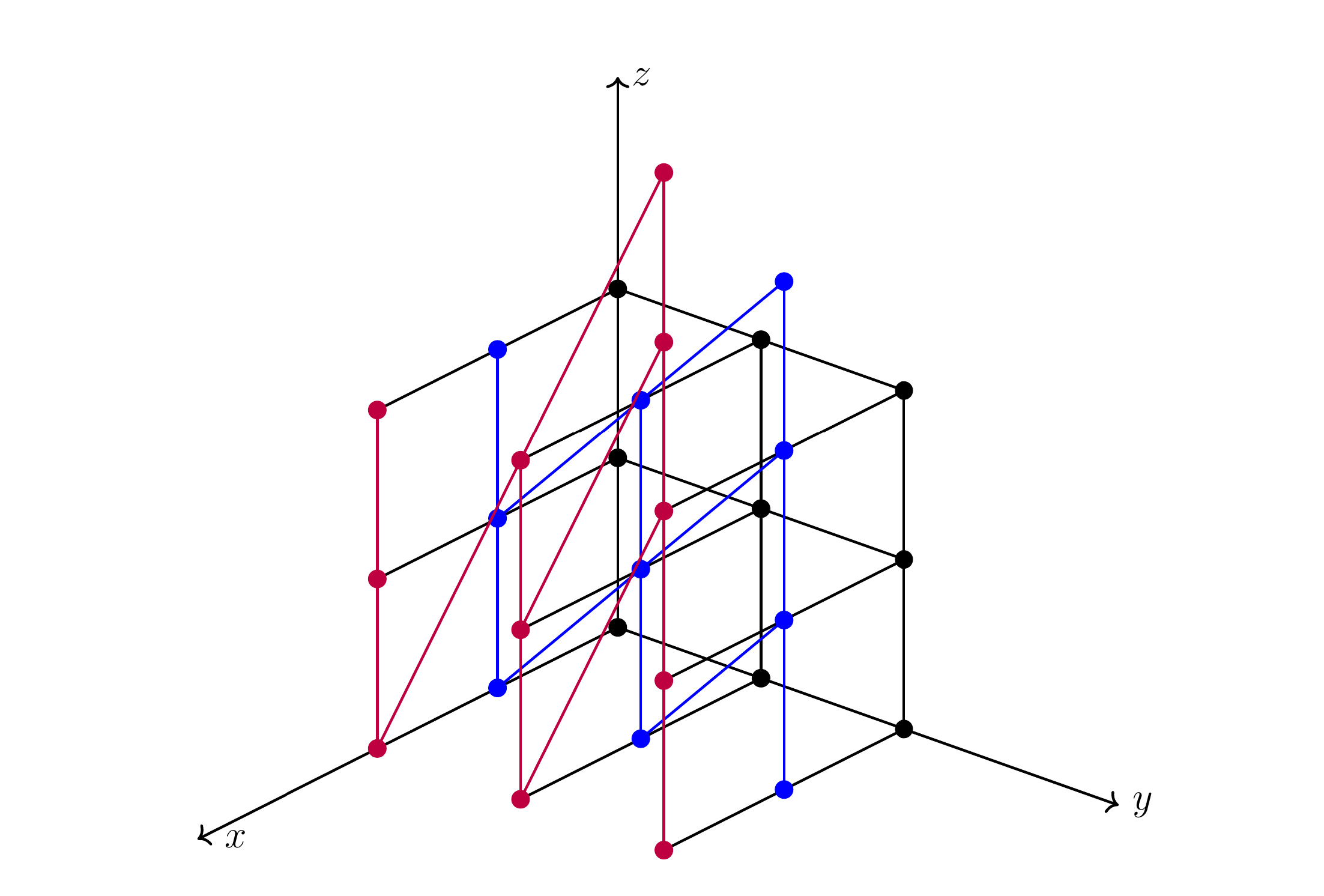}
\caption{A finite subgraph of the Cayley graph of the Heisenberg group $H(\Z)$. The colors are only to maintain visual clarity.}
\end{centering}
\end{figure}

The Heisenberg group is an example of a nonabelian nilpotent group. It has generators $$a=\begin{pmatrix}
1 & 1 & 0\\
0 & 1 & 0\\
0 & 0& 1
\end{pmatrix}\textrm{ and }b=\begin{pmatrix}
1 & 0 & 0\\
0 & 1 & 1\\
0 & 0& 1
\end{pmatrix}.$$ We add $$c=aba^{-1}b^{-1}=\begin{pmatrix}
1 & 0 & 1\\
0 & 1 & 0\\
0 & 0& 1
\end{pmatrix}$$ to the set of generators and consider the Cayley graph of $H(\Z)$ with generating set $S=\set{a, a^{-1}, b, b^{-1}, c, c^{-1}}$. Specifically, we set $V=H(\Z)$ and define $f\sim g$ when there exists $s\in S$ with $f=gs$. We also define $\mu_{fg}=1$ if and only if $f\simeq g$, and $\mu_{fg}=0$ otherwise. \\


The weighted graph $(H(\Z), \mu)$ satisfies HK($2$) as a consequence of the results of Hebisch and Saloff-Coste \cite{hebischsaloffcoste1}. It is also well known that $V_\mu(x,r)\asymp r^4$ which is uniform in $x\in H(\Z)$ and $r\geq 1$. Therefore $(H(\Z), \mu)$ satisfies HK$(4, 2)$. Therefore, with $o=$ the identity matrix, for any graph potential $\fct{W}{H(\Z)}{\Real}$ with $\abs{W(x)}\preceq \brac{x}^{-(2+\epsilon)}$ and such that $\Delta_\mu+W$ is subcritical, the heat kernel $p_\mu^W(x,y,n)$ satisfies two-sided Gaussian type estimates of the form of HK(4,2).\\

See also Hebisch \cite{hebisch1} for closely related results on heat kernels of Lie groups.

\begin{example}
Graphical Sierpinski Carpets
\end{example}

The following presentation is the same as in Barlow and Bass \cite{barlowbass2}. A graphical Sierpinski carpet is the graph analog of a Sierpinski carpet fractal and can be defined as follows. Let $H_0=\Z_+^d$ and for $x=(n_1,\dots,n_d)\in H_0$ write the ternary expansions $n_i=\sum_{r\geq 0} n_{ir}3^r$, where $n_{ir}\in \set{0,1,2}$. Define $$J_k=\set{(n_1, \dots, n_d)\colon n_{ik}=1 \textrm{ for } 1\leq i\leq d}.$$ This is union of disjoint cubes in $\Z^d$ of side length $3^k$. Let $$H_n=H_0\setminus \bigcup_{k=1}^n J_k$$ $$V^{(d)}=\bigcap_{n=0}^\infty H_n.$$ For $x,y\in V^{(d)}$ define $\mu_{xy}=1$ if and only if $\norm{x-y}_2\leq 1$ and $\mu_{xy}=0$ otherwise. This yields a weighted graph $(V^{(d)}, \mu)$.\\

\begin{figure}[hbt!]
\begin{centering}
\includegraphics[scale=0.3]{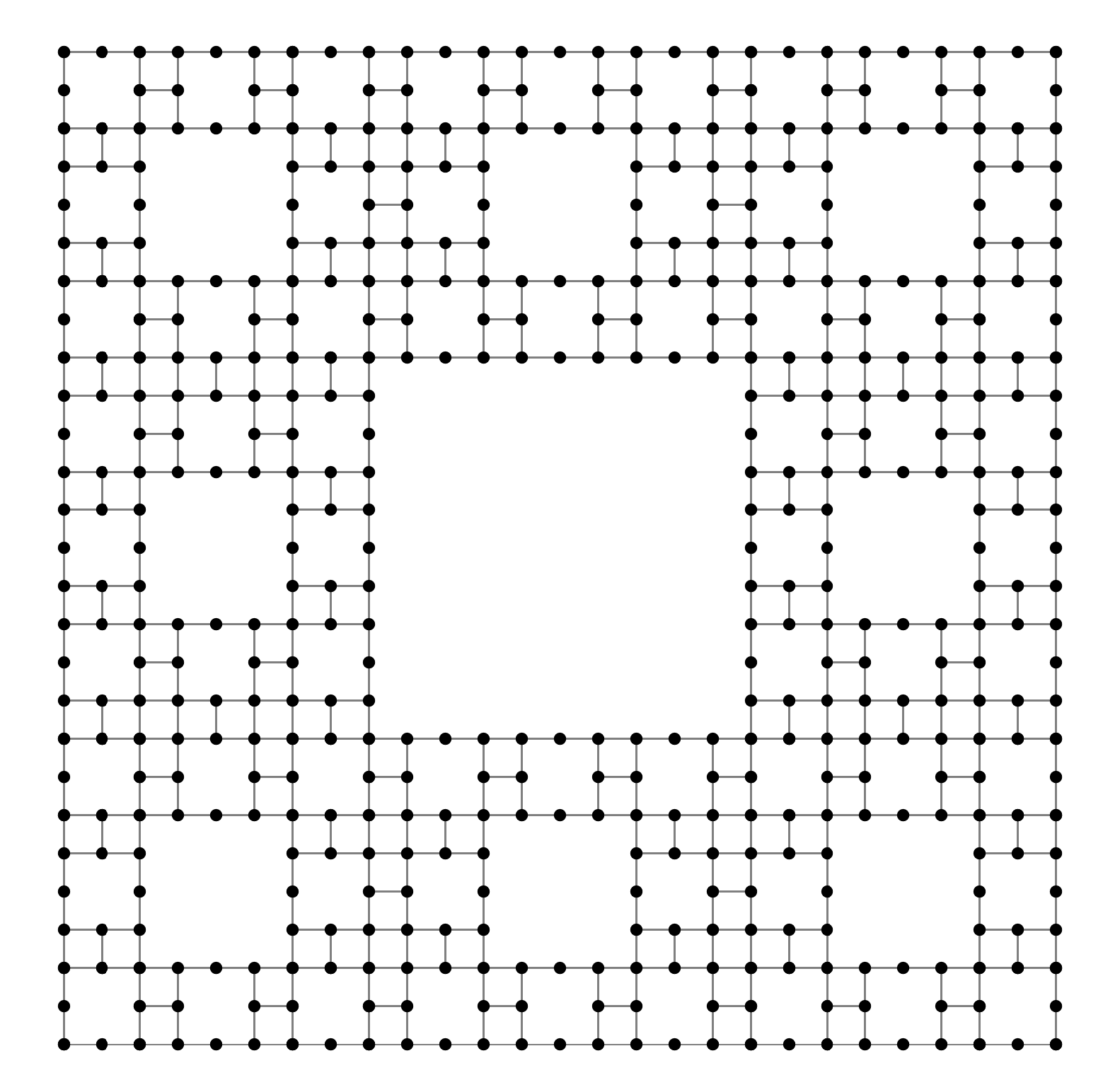}
\caption{A finite subgraph of a planar graphical Sierpinski carpet $V^{(2)}$.}
\end{centering}
\end{figure}

 See Figure 2 for a graphical Sierpinski carpet in the plane. Note that the pictured graph $V^{(2)}$ is not transient. We have chosen to include it over higher dimensional graphs for visual simplicity.\\

By Theorem 1.5 in Barlow and Bass \cite{barlowbass2}, for each $d\geq 2$ there exists $\beta(d)>2$ such that $(V^{(d)}, \mu)$ satisfies HK$(\beta(d))$. Furthermore, $V_\mu(x, r)\asymp r^{\alpha(d)}$ where $\alpha(d)=\log(3^d-1)/\log 3$, and therefore $(V^{(d)}, \mu)$ satisfies HK($\alpha(d), \beta(d)$). Note that $\alpha(d)$ is the Hausdorff dimension of the Sierpinski carpet fractal. It is known that $(V^{(d)}, \mu)$ is transient for $d\geq 3$.\\

Therefore, by Lemma \ref{katodecay} and Theorem \ref{heatkerneltheorem}, for $d\geq 3$, if $\fct{W}{V^{(d)}}{\Real}$ is a graph potential such that $\Delta_\mu+W$ is subcritical with $\abs{W(x)}\preceq \brac{x}^{-(\beta(d)+\epsilon)}$ for some $\epsilon>0$, then the heat kernel of $\Delta_\mu+W$ has two-sided sub-Gaussian estimates of the same form as $\Delta_\mu$.\\

One can also consider more general types of graphical Sierpinski carpets with similar results. See Barlow and Bass \cite{barlowbass2} for more details.

\begin{example}
Graphs satisfying HK($\alpha, \beta$) for prescribed $\alpha>\beta\geq 2$
\end{example}

In \cite{barlow2}, Barlow demonstrated that if $\alpha>\beta\geq 2$, then there exists a weighted graph $(V, \mu)$ satisfying HK($\alpha, \beta$). It then follows that Theorem \ref{heatkerneltheorem} yields estimates for the heat kernel of $\Delta_\mu+W$, provided this operator is subcritical and $W$ is a graph potential with $\abs{W(x)}\preceq \brac{x}^{-(\beta+\epsilon)}$ for some $\epsilon>0$.\\

\bibliographystyle{plain}

\bibliography{SubcriticalGraphBib.bib}

@book{aizenmanwarzel1,
  title={Random Operators},
  author={Aizenman, M. and Warzel, S.},
  publisher={Graduate Studies in Mathematics},
  year={2015}
}

@article{ancona1,
  title={{Green's functions, generalized first eigenvalues and perturbations of diffusions or Markov chains}},
  author={Ancona, A.},
  journal={Symposia Mathematica},
  volume={39},
  pages={1-25},
  year={1999}
}

@book{barlow1,
  title={Random Walks and Heat Kernels on Graphs},
  author={Barlow, M.},
  publisher={London Mathematical Society Lecture Note Series},
  year={2017}
}

@article{barlow2,
  title={Which values of the volume growth
and escape time exponent
are possible for a graph?},
  author={Barlow, M.},
  journal={Rev. Mat. Iberoamericana},
  volume={20},
  pages={1-31},
  year={2004}
}

@article{barlow3,
  title={{Analysis on the Sierpinski carpet}},
  author={Barlow, M.},
  journal={CRM Proceedings and Lecture Notes},
  volume={56},
  year={2013}
}

@article{barlowbass1,
  title={{Stability of Parabolic Harnack Inequalities}},
  author={Barlow, M. and Bass, R.},
  journal={Trans. Amer. Math. Soc.},
  volume={356(4)},
  pages={1501-1533},
  year={2003}
}

@article{barlowbass2,
  title={{Random walks on graphical Sierpinski carpets}},
  author={Barlow, M. and Bass, R.},
  journal={Random Walks and Discrete Potential Theory},
  volume={39},
  pages={26-55},
  year={1999}
}

@article{barlowbass3,
  title={{Uniqueness of Brownian motion on Sierpinski carpets}},
  author={Barlow, M. and Bass, R. and Kumagai, T. and Teplyaev, A.},
  journal={Journal Europ. Math. Soc.},
  volume={12},
  pages={655-701},
  year={2010}
}

@article{chen1,
  title={{Gaugeability and Conditional Gaugeability}},
  author={Chen, Z.-Q.},
  journal={Transactions of the American Mathematical Society},
  volume={354 (11)},
  pages={4639-4679},
  year={2002}
}

@article{delmotte1,
  title={{Parabolic Harnack inequality and estimates of Markov chains on graphs}},
  author={Delmotte, T.},
  journal={Rev. Math. Iberoamericana},
  volume={15},
  pages={181-232},
  year={1999}
}

@article{gravesmccleary1,
  title={{Heat Kernel Estimates for Schr\"odinger Operators with Decay at Infinity on Parabolic Manifolds}},
  author={Graves-McCleary, A. and Saloff-Coste, L.},
  journal={Trans. Amer. Math. Soc.},
  volume={Published electronically},
  pages={1-29},
  year={2026},
  doi={https://doi.org/10.1090/tran/9732}
}

@article{gravesmccleary2,
  title={{The boundary Harnack principle and the $3G$ principle in fractal-type spaces}},
  author={Graves-McCleary, A. and Saloff-Coste, L.},
  journal={Mathematische Nachrichten},
  volume={298(11)},
  pages={3554-3575},
  year={2025}
}

@book{grigoryan1,
  title={Introduction to Analysis on Graphs},
  author={Grigor'yan, A.},
  publisher={American Mathematical Society University Lecture Series},
  year={2018}
}

@article{grigoryan2,
author = {Grigor'yan, A.},
year = {2006},
month = {01},
pages = {93-191},
title = {Heat kernels on weighted manifolds and applications},
volume = {398},
isbn = {9780821836989},
journal = {Contemporary Mathematics},
doi = {10.1090/conm/398/07486}
}

@article{grigoryanhu,
  title={{Heat kernels and Green functions on metric measure spaces}},
  author={Grigor'yan, A. and Hu, J.},
  journal={Canad. J. Math.},
  volume={66(3)},
  pages={641-699},
  year={2014}
}

@article{grigoryantelcs1,
  title={{Sub-Gaussian estimates of heat kernels on infinite graphs}},
  author={Grigor'yan, A. and Telcs, A.},
  journal={Duke Math J.},
  volume={109 no. 3},
  pages={451-510},
  year={2001}
}

@article{hebischsaloffcoste1,
  title={{Gaussian estimates for Markov chains and random walks on groups}},
  author={Hebisch, W. and Saloff-Coste, L.},
  journal={Annals of Probability},
  volume={21 no. 2},
  pages={673-709},
  year={1993}
}

@article{hebisch1,
  title={{On heat kernels on Lie groups}},
  author={Hebisch, W.},
  journal={Math. Z.},
  volume={210},
  pages={593-605},
  year={1992}
}

@book{kigami1,
  title={Analysis on Fractals},
  author={Kigami, J.},
  publisher={Cambridge Tracts University Press},
  year={2001}
}

@article{kusuokazhou1,
  title={{Dirichlet forms on fractals: Poincar\'e constant and resistance}},
  author={Kusuoka, S. and Zhou, X. Y.},
  journal={Probab. Theory Relat. Fields},
  volume={93},
  pages={169-196},
  year={1992}
}

@book {lawlerlimic1,
    AUTHOR = {Lawler, Gregory F. and Limic, Vlada},
     TITLE = {Random walk: a modern introduction},
    SERIES = {Cambridge Studies in Advanced Mathematics},
    VOLUME = {123},
 PUBLISHER = {Cambridge University Press, Cambridge},
      YEAR = {2010},
     PAGES = {xii+364},
      ISBN = {978-0-521-51918-2},
   MRCLASS = {60G50 (60-02)},
  MRNUMBER = {2677157},
MRREVIEWER = {Andrew\ R.\ Wade},
       DOI = {10.1017/CBO9780511750854},
       URL = {https://doi.org/10.1017/CBO9780511750854},
}

@book{spitzer1,
  title={Principles of random walk},
  author={Spitzer, Frank},
  volume={34},
  year={2013},
  publisher={Springer Science \& Business Media}
}

@article{takeda1,
  title={{Gaussian Bounds of Heat Kernels for Schr\"odinger Operators on Riemannian Manifolds}},
  author={Takeda, M.},
  journal={Bull. London Math. Soc.},
  volume={39},
  pages={85-94},
  year={2007}
}

@book{woess1,
  title={Random Walks on Infinite Graphs and Groups},
  author={Woess, W.},
  publisher={Cambridge University Press},
  year={2009}
}

\end{document}